\documentclass{article}

\usepackage{anyfontsize,fix-cm}
\usepackage[T1]{fontenc}

\usepackage[english]{babel}
\usepackage{amssymb, amsthm, mathtools, microtype,xcolor} 
\usepackage{mathabx}

\usepackage[backref, pdfencoding=auto,  psdextra]{hyperref}

\usepackage[abbrev, nobysame]{amsrefs} 

\renewcommand{\MR}[1]{} 
\renewcommand{\PrintDOI}[1]{}

\usepackage{tikz-cd}

\newtheorem*{theorem*}{Theorem}
\newtheorem{theorem}{Theorem}[section]
\newtheorem{lemma}[theorem]{Lemma}

\newtheorem{proposition}[theorem]{Proposition}

\theoremstyle{definition}
\newtheorem*{example}{Example}

\theoremstyle{remark}
\newtheorem*{remark}{Remark}

\DeclarePairedDelimiter\ceil{\lceil}{\rceil}
\DeclarePairedDelimiter\floor{\lfloor}{\rfloor}
\DeclarePairedDelimiter\abs{\lvert}{\rvert}

\newcommand{\gt}{\tau} 

\newcommand{\into}{\hookrightarrow}

\title{$L^2$-Betti numbers of branched covers of hyperbolic manifolds}
\author{Grigori Avramidi\thanks{The first author thanks the Max Planck Institut f\"ur Mathematik for its hospitality and financial support.}, Boris Okun\thanks{The second author was partially supported by Simons Foundation grant MPS-TSM-00007773}, and Kevin Schreve\thanks{The third author was partially supported by the NSF grant DMS-2203325}}

\begin{document}

\maketitle
\begin{abstract}
    We show that Gromov--Thurston branched covers satisfy the Singer conjecture whenever the degree of the cover is not divisible by a finite set of primes determined by the base manifold and the branch locus.
\end{abstract}
\section{Introduction}
The Singer conjecture predicts vanishing of $L^2$-Betti numbers of closed aspherical manifolds outside the middle dimension.
In particular, for odd dimensional manifolds it predicts that all $L^2$-Betti numbers vanish.
On page 152 of \cite{gromovasymptotic}, Gromov discusses this conjecture and remarks that ``one cannot exclude a counterexample among (strongly pinched) ramified coverings of closed $(2k+1)$-dimensional manifolds of constant negative curvature''. Such branched covers (both strongly pinched and not) were constructed by Gromov and Thurston in \cite{gromovthurston}.
In the odd dimensional case, strongly pinched negative curvature can be used to show that the $L^2$-Betti numbers vanish outside the middle two dimensions using analytic methods \cite{donnellyxavier}, but this still leaves open the possibility that those two middle $L^2$-Betti numbers may be non-zero.
In this paper we use the skew field approach to $L^2$-Betti numbers together with special cube complex technology to prove the Singer conjecture for some of these branched covers.

\subsection*{Gromov--Thurston branched covers} 
We begin by recalling a method from \cite{gromovthurston} that constructs a family of cyclic branched covers of certain hyperbolic manifolds.

Let $M^n$ be a closed, oriented, hyperbolic $n$-manifold with two totally geodesic, (possibly disconnected) hypersurfaces $V_1,V_2$ which intersect transversely.
\begin{example}
    The group of automorphisms of the quadratic form $-\sqrt{2}x_{0}^2+x_1^2+\cdots+x_n^2$ over the ring of integers of the field $\mathbb Q(\sqrt 2)$ acts properly and cocompactly by isometries on hyperbolic space $\mathbb H^n$.
    It has a finite index normal subgroup $\Gamma$ that is torsion-free and acts by orientation preserving isometries.
    The quotient $M=\mathbb H^n/\Gamma$ is a closed, orientable, hyperbolic manifold with fundamental group $\Gamma$.
    It has orthogonal, embedded, totally geodesic  hypersurfaces $V_1$ and $V_2$, where $V_i$ is covered by the $x_i=0$ hyperplane $\mathbb H^{n-1}_{i}$ in $\mathbb H^n$.
    More generally, any hyperbolic manifold of simple type (these are, up to commensurability, the manifolds defined by quadratic forms) has a finite cover with a pair of orthogonal hypersurfaces.
    In odd dimensions $\neq 3,7$, all arithmetic hyperbolic manifolds are of simple type (see Remark 10.6 in \cite{hw}.)
\end{example}
Passing to a finite cover if necessary, we can make sure that the hypersurfaces $V_i$ are separating in the sense that---for each $i$---$M$ decomposes as a union of two compact manifolds with boundary $V_i$, glued along the boundary.\footnote{If $V_1$ is not separating, then intersecting with $V_1$ gives a surjective homomorphism $\pi_1M\rightarrow \mathbb Z/2$.
In the corresponding double cover $M'\rightarrow M$, the inverse image $V_1'$ of $V_1$ is separating because its intersection number with elements of $\pi_1M'$ is even.
If the inverse image $V_2'$ of $V_2$ does not separate $M'$, repeat the argument.}

Since $V_1$ and $V_2$ are separating hypersurfaces in the orientable manifold $M$, they are orientable.
Since they are transverse, the intersection $V=V_1\cap V_2$ is a codimension two submanifold with trivial normal bundle.
Let $M_0=M-(V\times \mathbb D^2)$.
Since $V_2$ separates $M$, $V$ separates $V_1$ into $V_1=V_1^+\cup_VV_1^{-}$, so $V$ is the boundary of an orientable $(n-1)$-dimensional submanifold $V_1^+$ of $M$.
Intersection with $V_1^+$ gives a surjective homomorphism $\phi:\pi_1(M_0)\rightarrow\mathbb Z$ which, for every positive integer $d$, defines a $d$-fold cyclic cover $M_0'\rightarrow M_0$.
Restricted to the boundary, the cover is $V\times\partial\mathbb D^2\rightarrow V\times\partial\mathbb D^2$ with the map being identity on the first factor and degree $d$ on the second factor.
Hence gluing in $V\times \mathbb D^2$ along the boundary of the manifold $M_0'$ gives the $d$-fold cyclic branched cover $\widehat M\rightarrow M$.

The hyperbolic manifolds of simple type have the additional property that they have finite covers with special\footnote{We call a group \emph{special} if it is the fundamental group of a compact special cube complex in the sense of \cite{hw}.}
fundamental group \cites{hw,bhw}.
Giralt showed in \cite{giralt} that when $M$ has special fundamental group, the cyclic branched cover $\widehat M$ does, as well.
We use her result to prove the following theorem.

\begin{theorem*}
    Suppose $M$ is a closed, orientable, hyperbolic $n$-manifold with virtually special fundamental group and $V_1,V_2$ are two separating, totally geodesic hypersurfaces in $M$ intersecting transversely in $V=V_1\cap V_2$.
    Then, there is a positive integer $m$ (determined by $M,V_1$ and $V_2$) so that for $d$ relatively prime to $m$, the $d$-fold cyclic branched cover $\widehat M\rightarrow M$ satisfies the Singer conjecture:
\[ 
b^{(2)}_{\neq n/2}(\widehat M)=0. 
\]
\end{theorem*}
\begin{remark}
    It is shown in \cite{gromovthurston} that the branched covers $\widehat M$ have metrics of negative curvature, but for fixed $M$ and $V$ and large enough $d$, they do not admit metrics of constant negative curvature.
\end{remark}
\begin{remark}
For even dimensional Gromov--Thurston branched covers, there is a proof of the Singer conjecture that works for any degree, and doesn't require any assumptions on the fundamental group of $M$. We describe it in Section \ref{remarks}.
\end{remark}

\section{Proof for prime power branched covers\label{primeproof}
} In this section we give a proof of the theorem in the special case when the degree of the branched cover is a prime power.
In outline, we will first translate analytic statements about $L^2$-Betti numbers of the hyperbolic manifolds $M$ and $V$ to skew field Betti numbers for a certain skew field $D$ (associated to $\pi_1M$) of prime characteristic $p$, then proceed to compute these $D$-Betti numbers for the complement $M_0$, its $\mathbb Z/p^r$-cover $M_0'$ and the branched cover $\widehat M$, before finally translating back to the desired statement about $L^2$-Betti numbers of $\widehat M$.

\subsection*{Reduction to special fundamental groups} 
By assumption, $M$ has a finite cover $\pi:M'\rightarrow M$ whose fundamental group is special.
Let $V_i'=\pi^{-1}(V_i)$ be the preimage of $V_i$ under the covering map.
Note that $V_i'$ separates $M'$ since $V_i$ separates $M$.
Moreover, $V'_1\cap V'_2$ is the preimage $\pi^{-1}(V)$ of $V$, and we will denote it $V'$.
Therefore, we can form the $d$-fold cyclic branched cover $\widehat M'$ of $M'$ branched along $V'$ as described on page 2. Its defining homomorphism $\phi':\pi_1(M'-V'\times\mathbb D^2)\rightarrow \mathbb Z$ is given by intersection with $V_1'^+$, so it factors through $\phi:\pi_1(M-V\times\mathbb D^2)\rightarrow\mathbb Z$ (given by intersection with $V_1^+$).
It follows that the branched cover $\widehat M'$ is the pullback of $\widehat M\rightarrow M$ via $M'\rightarrow M$, i.e.
it fits into the commutative square
\[
\begin{tikzcd}[sep=scriptsize] \widehat M' \rar \dar & \widehat M \dar\\
    M' \rar & M,
\end{tikzcd}
\]
where the vertical maps are branched covers of degree $d$ and the horizontal maps are covers of degree $\abs{M'\rightarrow M}$.
In particular, $\widehat M'$ is a finite cover of $\widehat M$, and by multiplicativity of $L^2$-Betti numbers we have $b^{(2)}_i(\widehat M)=b^{(2)}_i(\widehat M')/\abs{M'\rightarrow M}$.
Therefore, if the theorem is true for $M'$ then it is also true for $M$.
So, we may assume $M$ has special fundamental group and will do so for the rest of the paper.

\subsection*{$L^2$-Betti numbers via skew fields}
We use the skew field approach to $L^2$-Betti numbers, see e.g.
\cite{aos}*{Section 3} for additional details.
Let $G=\pi_1M$.
Since $G$ is special, it is a subgroup of a right-angled Artin group. 
This implies it is residually torsion-free nilpotent and hence bi-orderable.
Therefore, its group ring $\mathbb FG$ embeds in the (Malcev--Neumann) skew field of power series on $G$ with well-ordered support and coefficients in $\mathbb F$ for any field $\mathbb F$.
The division closure of $\mathbb FG$ in this skew field is independent of the choice of order (by \cite{hughes}) and we denote it by $D_{\mathbb FG}$.
Therefore, we can take homology of $M$ with local coefficients in $D_{\mathbb FG}$, and since $D_{\mathbb FG}$ is a skew field we have associated Betti numbers.
Corollaries 4.1 and 4.2 of \cite{aos} show that for large enough primes $p$ we have
\[
 b_i(M;D_{\mathbb F_pG})=b_i(M;D_{\mathbb QG})=b_i^{(2)}(M).
\]
Since $V$ is totally geodesic in $M$, each component $V_0$ of $V$ gives an inclusion $\pi_1V_0<G$.
Therefore $\pi_1V_0$ is residually torsion-free nilpotent and we can define $D_{\mathbb F\pi_1V_0}$ as before.
Applying Corollaries 4.1 and 4.2 of \cite{aos} to $V_0$ gives
\begin{equation}\label{component}
    b_i(V_0;D_{\mathbb F_p\pi_1V_0})=b_i(V_0;D_{\mathbb Q\pi_1V_0})=b_i^{(2)}(V_0)
\end{equation}
for large enough $p$.
If we use the order on $\pi_1V_0$ induced from $G$, then the division closure of $\mathbb F\pi_1V_0$ in $D_{\mathbb FG}$ agrees with $D_{\mathbb F\pi_1V_0}$, implying that 
\[
H_i(V_0;D_{\mathbb F\pi_1V_0})\otimes_{D_{\mathbb F\pi_1V_0}}D_{\mathbb FG}=H_i(V_0;D_{\mathbb FG})
\]
and consequently $b_i(V_0;D_{\mathbb F\pi_1V_0})=b_i(V_0;D_{\mathbb FG})$.
Substituting this into \eqref{component} and summing over the components of $V$ gives 
\[
 b_i(V;D_{\mathbb F_pG})=b_i(V;D_{\mathbb QG})=b_i^{(2)}(V).
\] 
for large enough $p$.
Fix a prime $p$ for which the equalities for $V$ and $M$ hold, and set $D:=D_{\mathbb F_pG}$ to conserve notation.

\subsection*{$D$-Betti numbers of the complement $M_0$} 
Since both $M$ and $V$ are closed hyperbolic manifolds, their $L^2$-Betti numbers vanish outside the middle dimension (\cite{dodziuk}), implying that
\begin{align}
    H_{\neq n/2}(M^n;D)&=0,\label{mvanish}
    \\
    H_{\neq n/2-1}(V^{n-2};D)&=0.\label{vvanish}
\end{align}
Excision and Poincar\'e duality implies
\begin{multline*}
    H_i(M,M_0;D)\cong H_i(V\times\mathbb D^2,\partial;D)\cong H^{n-i}(V\times\mathbb D^2;D) \\
    \cong H^{n-i}(V;D)\cong H_{i-2}(V;D)
\end{multline*}

and we conclude from \eqref{vvanish} that
\begin{equation}\label{relvanish}
    H_{\neq n/2+1}(M,M_0;D)=0.
\end{equation}
For later use, note that the same argument applied to the pair $(\widehat M,M'_0)$ implies
\begin{equation}\label{rel'vanish}
    H_{\neq n/2+1}(\widehat M,M'_0;D)=0.
\end{equation}
The long exact sequence for the pair $(M,M_0)$
\begin{equation*}
    \cdots\rightarrow H_{*+1}(M,M_0;D)\rightarrow H_*(M_0;D)\rightarrow H_{*}(M;D)\rightarrow\cdots
\end{equation*}
together with \eqref{mvanish} and \eqref{relvanish} imply
\begin{equation}\label{compvanish}
    H_{<n/2}(M_0;D)=0.
\end{equation}
\subsection*{$D$-Betti numbers of the $\mathbb Z/p^r$-cover $M'_0$} 
Next, we claim that $H_{<n/2}(M'_0;D)=0$.

To see this, we will use the interpretation of $D$-Betti numbers of $M_0$ and $M_0'$ as the infimum of normalized $\mathbb F_p$-Betti numbers over finite covers pulled back from $M$, and the fact that normalized $\mathbb F_p$-Betti numbers are monotone in $p$-power covers, both of which we recall next.
Denote the degree of a finite cover $Y'\rightarrow Y$ by $\abs{Y'\rightarrow Y}$.
\begin{itemize}
    \item {\bf Inf formula:} Given a finite complex $Y$, a residually torsion-free nilpotent group $\Gamma$, and a homomorphism $\pi_1Y\rightarrow \Gamma$, Theorem 3.6 of \cite{aos} implies
    \begin{equation}\label{infformula}
        b_i(Y;D_{\mathbb F\Gamma})=\inf_{Y'\rightarrow Y} {b_i(Y';\mathbb F)\over \abs{Y'\rightarrow Y}}.
    \end{equation}
    where the inf is over the finite covers $Y'\rightarrow Y$ pulled back from $B\Gamma$.
    \item {\bf p-monotonicity:} If $X'\rightarrow X$ is a regular cover of degree $p^r$, then Theorem 1.6 in \cite{blls} implies
    \[
    b_i(X';\mathbb F_p)\leq p^{r}\cdot b_i(X;\mathbb F_p).
    \]

\end{itemize}
Now, let $X\rightarrow M_0$ be a finite cover pulled back from $M$ and $X'\rightarrow M_0'$ its pullback to $M_0'$.
These covers fit into
\[
\begin{tikzcd}[sep=scriptsize] 
    X' \rar \dar & X \dar\\
    M_0' \rar & M_0
\end{tikzcd}
\]
where the horizontal maps are regular covers of degree $p^r$ and the vertical maps are covers of the same finite degree $\abs{X'\rightarrow M_0'}=\abs{X\rightarrow M_0}$.
Therefore 
\[ 
{b_i(X';\mathbb F_p)\over \abs{X'\rightarrow M'_0}}\leq p^r{b_i(X;\mathbb F_p)\over\abs{X\rightarrow M_0}},
\]
and we conclude from the inf formula for $D=D_{\mathbb F_pG}$-Betti numbers that 
\[
 b_i(M'_0;D)=\inf_{X'\rightarrow M'_0}{b_i(X';\mathbb F_p)\over\abs{X'\rightarrow M_0'}}\leq p^r\cdot\inf_{X\rightarrow M_0}{b_i(X;\mathbb F_p)\over\abs{X\rightarrow M_0}}=p^r\cdot b_i(M_0;D),
 \] 
 where the infs in both cases are over finite covers (of $M'_0$ on the left and $M_0$ on the right) that are pulled back from $M$.
Therefore, \eqref{compvanish} implies
\begin{equation}\label{comp'vanish}
    H_{<n/2}(M_0';D)=0,
\end{equation}
which finishes the proof of the claim.

\subsection*{$D$-Betti numbers of the branched cover $\widehat M$} 
Finally, we look at the long exact sequence of the pair $(\widehat M,M_0')$ corresponding to the branched cover
\begin{equation*}
    \cdots\rightarrow H_*(M_0';D)\rightarrow H_*(\widehat M;D)\rightarrow H_{*}(\widehat M,M_0';D)\rightarrow\cdots
\end{equation*}
and note that \eqref{comp'vanish} and \eqref{rel'vanish} imply $H_{<n/2}(\widehat M;D)=0$ and so, by Poincar\'e duality, also $H_{>n/2}(\widehat M;D)=0$.
In summary, we have shown that
\begin{equation}\label{branchedvanish}
    H_{\neq n/2}(\widehat M;D_{\mathbb F_{p}G})=0.
\end{equation}

\subsection*{$L^2$-Betti numbers of $\widehat M$} 
It remains to relate this vanishing to the $L^2$-Betti numbers of $\widehat M$.
To that end, we use the result of Giralt  \cite{giralt}*{Theorem 2} that $\pi_1\widehat M$ is special.
This implies that $\mathbb F\pi_1\widehat M$ embeds in a skew field $D_{\mathbb F\pi_1\widehat M}$ (defined as before) and that for $\mathbb F=\mathbb Q$ this skew field computes the $L^2$-Betti numbers of $\widehat M$.
The inf formula lets us relate the $D_{\mathbb F\pi_1\widehat M}$-Betti numbers to $D_{\mathbb F G}$-Betti numbers.
Altogether, we obtain the following monotonicity formula (see also Lemma 2.6 of \cite{aos2}) 
\[ 
b_i^{(2)}(\widehat M)=b_i(\widehat M;D_{\mathbb Q\pi_1\widehat M})\leq b_i(\widehat M;D_{\mathbb F_p\pi_1\widehat M})\leq b_i(\widehat M;D_{\mathbb F_pG}) 
\]
where the first equality is Corollary 4.2 of \cite{aos}, the first inequality follows from the inf formula \eqref{infformula} and the fact that $b_i(X;\mathbb Q)\leq b_i(X;\mathbb F_p)$, and the second inequality also follows from the inf formula since the left term is an inf over all finite covers of $\widehat M$ while the right term is an inf over only those finite covers that are pulled back from $M$.

So, we conclude from \eqref{branchedvanish} that $\widehat M$ satisfies the Singer conjecture.
This finishes the proof of the theorem for prime power covers.

\section{Remarks\label{remarks}}
\subsection*{Exceptional primes}
Multiplicativity of skew field Betti numbers \cite{aos}*{Lemma 3.3} shows the same prime $p$ works if we start with another pair $(M',V')$ commensurable to $(M,V)$, in the sense that $M$ and $M'$ have a common finite cover and $V$ and $V'$ have a common finite cover.
So, the exceptional primes to which the above argument does not apply are determined by commensurability classes of $M$ and $V$.
\subsection*{$\mathbb F_p$-Singer property in special covers}
Say that a closed $n$-manifold $N$ with special fundamental group satisfies the $\mathbb F_p$-Singer property if $H_{\neq n/2}(N;D_{\mathbb F_p\pi_1N})=0$.
Looking at what we used in the proof above, we observe that it gives the following: Suppose that
\begin{itemize}
    \item $M^n$ is a closed, orientable $n$-manifold,
    \item $V^{n-2}\subset M^n$ is a codimension two, $\pi_1$-injective, closed, orientable submanifold,
    \item $\widehat M\rightarrow M$ is a $p^r$-fold cyclic branched cover of $M$ branched over $V$, and
    \item both $\pi_1M$ and $\pi_1\widehat M$ are special.
\end{itemize}
If $M$ and each component of $V$ satisfy the $\mathbb F_p$-Singer property, then $\widehat M$ satisfies the $\mathbb F_p$-Singer property.

\subsection*{Singer conjecture in even dimensions}
When $M$ is even dimensional, there is another argument that makes more use of the hypersurfaces $V_i$, proves the Singer conjecture for branched covers of all degrees, and does not require special fundamental groups. 
In particular, Mayer--Vietoris for $L^2$-Betti numbers shows that for $\pi_1$-injective splittings $X_1\cup_ZX_2$ we have: 
\begin{itemize}
\item 
$b^{(2)}_k(Z)=0$ and $b^{(2)}_k(X_1\cup_Z X_2)=0$ implies $b_k^{(2)}(X_i)=0$,
\item 
$b^{(2)}_k(X_i)=0$ and $b^{(2)}_{k-1}(Z)=0$ implies $b^{(2)}_k(X_1\cup_Z X_2)=0$.
\end{itemize}
We have such splittings for the closed hyperbolic $n$-manifold $M=M^+\cup_{V_1}M^-$, and its hypersurface $V_1=V_1^+\cup_VV_1^-$. 
So, via the first bullet point, the Singer conjecture for the closed hyperbolic manifolds $M$ and $V_1$ implies that 
\begin{equation}
\label{M}
b^{(2)}_{>n/2}(M^{\pm})=0,
\end{equation}
while the Singer conjecture for the closed hyperbolic manifolds $V_1$ and $V$ implies 
\begin{equation}
\label{V}
b^{(2)}_{>(n-1)/2}(V_1^{\pm})=0.
\end{equation}
The $d$-fold branched cover $\widehat M$ can be decomposed as 
\[
\widehat M=M^+\cup_{V_1}M_d
\]
where $M_d$ is defined inductively by 
\[
M_1=M^-, \quad \text{ and }\quad M_{i+1}=M_i\cup_{V_1^+}M^+\cup_{V_1^-}M^-,
\]
and again all splittings appearing in this description are $\pi_1$-injective. 
For $k>\ceil{n/2}$ we have $k-1>(n-1)/2$, so these splittings together with \eqref{M}, \eqref{V}, and the Singer conjecture for $V_1$ imply (via the second bullet point) that $b^{(2)}_{k}(\widehat M)=0$. 
By Poincar\'e duality, we get the same conclusion for $k<\floor{n/2}$. 
When $n$ is even, this establishes the Singer conjecture for $\widehat M^n$, but when $n$ is odd it only shows that the $L^2$-Betti numbers vanish outside the middle two dimensions.
The analytic results of \cite{donnellyxavier} give the same conclusions under the additional curvature pinching assumption $-1\leq K\leq -({n-2\over n-1})^2$.

\subsection*{Some composite covers via iteration} 
Iterating the method in Section~\ref{primeproof} lets us establish the Singer conjecture for some Gromov--Thurston branched covers of composite degree.
Here is how it works in the simplest instance where there are two prime factors.
Denote by $\widehat M_k$ the branched cover of degree $k$.
If $d=p^rq^s$ for primes $p$ and $q$, then $\widehat M_{p^r}$ has special fundamental group and, for large $p$, our result shows that it satisfies the Singer conjecture.
Therefore, by Corollary 4.2 of \cite{aos}, it also satisfies $\mathbb F_q$-Singer for large $q$ (depending on $M,V$ \emph{and} $p^r$).
For such $q$, we can apply the same argument to $\widehat M_{p^rq^s}$, thought of as the $q^s$-degree branched cover of $\widehat M_{p^r}$ and conclude that $\widehat M_{p^rq^s}$ satisfies the Singer conjecture.
The $q$ for which this works depends on $p^r$, so we don't get the full theorem this way.

\section{Homology vanishing in cyclic covers} In order to prove the theorem for cyclic branched covers of other degrees $d$, we investigate when homology vanishing is preserved in $d$-fold cyclic covers.
We do that in this section and then return to the proof of the theorem in the next.
Our argument is inspired by Fox's result that branched $d$-fold cyclic covers of knot complements in $S^3$ are rational homology spheres when the Alexander polynomial does not vanish at any $d$-th roots of unity (6.2 in \cite{fox3}).
\subsection*{Setup}
Let $X$ be a finite complex. Suppose $\Gamma:=\pi_1X$ surjects onto $\mathbb Z$ via $\phi:\Gamma\rightarrow \mathbb Z$, let $\Gamma_d:=\ker(\Gamma\rightarrow\mathbb Z\rightarrow\mathbb Z/d)$, and denote by $X_d=\widetilde X/\Gamma_d$ the corresponding $d$-fold cyclic cover of $X$. Let $\psi:\mathbb Z\Gamma\rightarrow D$ be a ring homomorphism to a skew field.

\subsection*{Long exact sequence for $H_*(X_d;D)$} 
Let $R:=D[\tau,\tau^{-1}]$ be the Laurent polynomial ring with coefficients in $D$.
It is a ring over $\mathbb Z[\Gamma]$ via the homomorphism $\mathbb Z[\Gamma]\rightarrow D[\tau,\tau^{-1}]$ induced by $g\mapsto\psi(g)\tau^{\phi(g)}$.
Since $\tau$ is central, we have a short exact sequence of $R$-bimodules: 
\begin{equation}\label{Rmodules}
    0\rightarrow R\xrightarrow{(\tau^d-1)\cdot} R\rightarrow R/(\tau^d-1)\rightarrow 0
\end{equation}
Applying $\otimes_{\mathbb Z\Gamma}C(\widetilde X)$ to this sequence gives a short exact sequence of chain complexes of left $R$-modules 
\[ 
0\rightarrow C_*(X;R) \xrightarrow{(\tau^d-1)\cdot} C_*(X;R)\rightarrow C_*(X;R/(\tau^d-1))\rightarrow 0 
\] 
and associated long exact homology sequence 
\[
\cdots\rightarrow H_*(X;R)\xrightarrow{(\tau^d-1)\cdot} H_*(X;R)\rightarrow H_*(X;R/(\tau^d-1))\rightarrow\cdots 
\] 
The third term computes the $D$-homology of the $d$-fold cyclic cover $X_d$:

\begin{lemma}
    We have an isomorphism of left $D$-modules 
    \[
    H_*(X_d;D)\cong H_*(X;R/(\tau^d-1)).
    \]
\end{lemma}
\begin{proof}
    The homology group $H_*(X_d;D)$ is computed from the chain complex
    \begin{equation}\label{chainiso}
        D\otimes_{\mathbb Z\Gamma_d}C_*(\widetilde X)\cong(D\otimes_{\mathbb Z\Gamma_d}\mathbb Z\Gamma)\otimes_{\mathbb Z\Gamma}C_*(\widetilde X).
    \end{equation}
    To prove the lemma, we need to identify the $D-\mathbb Z\Gamma$-module in parentheses with $R/(\tau^d-1)=D[\tau]/(\tau^d-1)$.
    We do this via the map

\begin{align*}
D\otimes_{\mathbb Z\Gamma_d}\mathbb Z\Gamma&\rightarrow D[\tau]/(\tau^d-1)\\
x\otimes g&\mapsto x\psi(g)\tau^{\phi(g)}.
\end{align*}
It is easy to check that this map is well-defined and invertible, with inverse given by $x\tau^i\mapsto x\psi(t^{-i})\otimes t^i$, where $t\in\Gamma$ is any\footnote{Different choices of $t$ give the same map.} element with $\phi(t)=1\in\mathbb Z$. 
\end{proof}

\subsection*{Structure of $D[\tau,\tau^{-1}]$-modules}
Since $D$ is a skew field, Theorem 1.3.2 in \cite{cohnbook} shows that $R=D[\tau,\tau^{-1}]$ is a left and right principal ideal domain.
This implies (Theorem 1.4.10 in \cite{cohnbook}) that
\[
H_*(X;R)\cong \bigoplus_{i = 1}^n {R\over R p_i(\tau)}.
\]

\subsection*{Central roots}
For a non-zero element $z\in D^*$, define the `evaluation at $z$ map' by
\begin{align*}
    ev_z:D[\tau,\tau^{-1}]&\rightarrow D\\
    \sum a_i \tau^i &\mapsto \sum a_i z^i.
\end{align*}
We denote $ev_z(p(\tau))$ by $p(z)$.
Note that when $z$ is central in $D$, then the map $ev_z$ is a ring homomorphism.
The non-zero roots of $p(\tau)$ are elements $z\in D^*$ with $p(z)=0$.
The degree of a non-zero Laurent polynomial $a_{-m}\gt^{-m} + \dots + a_n\gt^n$ is defined to be $m+n$.
This satisfies 
\[ 
\deg(p(\gt)q(\gt)) = \deg(p(\gt)) + \deg(q(\gt)).
\]
\begin{lemma}\label{centralroots}
    A degree $n$ Laurent polynomial $p(\tau)$ in $R$ has at most $n$ non-zero, central roots in $D$.
\end{lemma}
\begin{proof}
    This is obvious if $n = 1$, so suppose we know the statement for Laurent polynomials of degree $(n-1)$, and suppose $p(z)=0$ for some non-zero central element $z$.
    Apply polynomial long division to write $p(\tau)=q(\tau)(\tau-z)+c$ for some constant $c\in D$.
    Evaluating at $z$ implies $c=0$.
    So, we have $p(\tau)=q(\tau)(\tau-z)$, where $q(\tau)$ is a Laurent polynomial of degree $(n-1)$.
    Moreover, if $p(z')=0$ for some non-zero, central element $z'\neq z$, then $0=p(z')=q(z')(z-z')$ implies that $q(z')=0$.
    Therefore, there are $\le n-1$ such $z'$, and the statement follows.
\end{proof}
\begin{remark}
    There can be more non-central roots.
    For instance, the quaternions form a skew field that is a $4$-dimensional $\mathbb R$-vector space $\mathbb R\oplus\mathbb Ri\oplus\mathbb Rj\oplus\mathbb Rij$ with multiplication specified by $i^2=j^2=-1, ij=-ji$.
    In this skew field, all the elements $\{\pm 1,\pm i,\pm j,\pm ij\}$ (and their conjugates) are roots of the quadratic polynomial $x^2+1$.
    In general, Gordon and Motzkin \cite{gordonmotzkin} showed that the roots of $p(\gt)$ fall into $\le n$ conjugacy classes of $D$, and each conjugacy class contains either $0$, $1$, or infinitely many roots.
\end{remark}
\begin{lemma}\label{l:mult}
    Let $z$ be a non-zero, central element in $D$.
    If $p(z)\neq 0$, then
    \[ {R\over Rp(\tau)} \xrightarrow{(\tau-z)\cdot} {R\over Rp(\tau)}
    \]
    is an isomorphism.
\end{lemma}
\begin{proof}
    Since $z$ is central in $D$, the map $(\tau-z)\cdot$ is left $R$-linear.
    Since $p(z)\neq 0$, $\tau-z$ does not divide $p(\tau)$ and hence, the left ideal generated by $p(\tau)$ and $\tau-z$ is the full ring.
    Therefore, the map $(\tau-z)\cdot$ is onto.
    Since the domain and range are finite-dimensional $D$-vector spaces of the same dimension, this implies the map is an isomorphism.
\end{proof}

\subsection*{Vanishing results} 
Now we can prove the main vanishing result of this section.
\begin{proposition}\label{finite}
    Suppose $D$ is a skew field whose center contains $\mathbb C$.
    Then there is a positive integer $m$,\footnote{It follows from the proof that $m$ is determined by $H_k(X;D[\tau,\tau^{-1}])\oplus H_{k-1}(X;D[\tau,\tau^{-1}])$.}
    such that for any $d$ relatively prime to $m$, we have $H_k(X_d;D)=0$ if and only if $H_k(X;D)=0$.
\end{proposition}
\begin{proof}
    Break homology into three parts
    \[ 
    H_j(X;R)\cong R^{n_1} \oplus \bigoplus_{i=1}^{n_2} {R\over R(\tau-1)q_i(\tau)} \oplus \bigoplus_{i = 1}^{n_3} \frac{R}{Rp_i(\tau)} 
    \] 
    where the $q_i(\tau)$ are non-zero and the $p_i(\tau)$ are not divisible by $\tau-1$.
    Note that:
    \begin{enumerate}
        \item Multiplication by $\tau^d-1$ (in particular, by $\tau-1$) on the first factor is always injective and never surjective (unless $n_1=0$).
        \item Multiplication by $\tau-1$ on the second factor is neither injective nor surjective (unless $n_2=0$) and hence the same is true for $\tau^d-1$.
        \item By Lemma \ref{centralroots}, there is a finite set $S$ of roots of unity in $\mathbb C$ that occur as roots of $p_1(\tau),\dots, p_{n_3}(\tau)$.
        Let $m_j$ be the product of the orders of elements of $S$.
        If $d$ is relatively prime to $m_j$, then the $d$-th roots of unity $\{1,e^{2\pi i/d},\dots,e^{2\pi i(d-1)/d}\}$ are not in $S$, since
        \begin{itemize}
            \item $1$ is not in $S$ because the $p_i(\tau)$ are not divisible by $\tau-1$, and
            \item the non-trivial $d$-th roots of unity are not in $S$ because their orders are non-trivial factors of $d$, and hence do not divide $m_j$.
        \end{itemize}
        So, $p_i(e^{2\pi i l/d})\neq 0$ for all $p_i,l$, and we see by Lemma \ref{l:mult} that multiplication by $\tau^{d}-1=\prod_{l=1}^d (\tau-e^{2\pi i l/d})$ is an isomorphism on the third factor.
    \end{enumerate}
    We conclude that if $d$ is relatively prime to $m_j$, then multiplication by $\tau^d-1$ is injective (respectively surjective) on $H_j(X;R)$ if and only if $\tau-1$ is.
    Therefore, the long exact homology sequences for $X_1$ and $X_d$ imply that $H_k(X_d;D)=0$ if and only if $H_k(X_1;D)=0$ as long as $d$ is relatively prime to $m_k m_{k-1}$.
\end{proof}

For the sake of comparison, here is an $\mathbb F_p$-vanishing result that can be used to establish the main claim of Section~\ref{primeproof}.
\begin{proposition}
    Suppose $D$ has characteristic $p$.
    Then $H_k(X;D)=0$ if and only if $H_k(X_{p^r};D)=0$.
\end{proposition}
\begin{proof}
    In characteristic $p$ we have the equation $\tau^{p^r}-1=(\tau-1)^{p^r}$.
    It implies that multiplication by $\tau^{p^r}-1$ is injective (respectively surjective) if and only if multiplication by $\tau-1$ has the same property.
    By the long exact homology sequence, $H_k(X_{p^r};D)$ vanishes if and only if $\tau^{p^r}-1$ is surjective on $H_k(X;R)$ and injective on $H_{k-1}(X;R)$, so we conclude that $H_k(X_{p^r};D)$ vanishes if and only if $H_k(X_1;D)$ does.
\end{proof}

\section{Proof for general branched covers}
First, recall the statement we are proving.
\begin{theorem*}
    There is a positive integer $m$ (determined by $M,V_1$ and $V_2$) so that if $d$ is relatively prime to $m$, the $d$-fold cyclic branched cover $\widehat M$ satisfies the Singer conjecture: 
    \[ 
    b^{(2)}_{\neq n/2}(\widehat M)=0. 
    \]
\end{theorem*}
\begin{proof}
    Let $G=\pi_{1} M$.
    We use the skew field $D:=D_{\mathbb CG}$ since its center contains $\mathbb C$.
    It is clear (from the inf formula, for instance) that $b_*(X; D)=b_*(X;D_{\mathbb QG})$ so this skew field works just as well for computing $L^2$-Betti numbers.
    We want to apply Proposition~$\ref{finite}$.
    To do this, set $X=M_0=M-(V_1\cap V_2\times\mathbb D^2)$, $\Gamma=\pi_1M_0,$ and let $\phi:\pi_1M_0\rightarrow\mathbb Z$ be the surjective homomorphism used to defined the branched covers.
    The skew field $D$ is a right $\mathbb Z\Gamma$-module via the homomorphism $\psi:\mathbb Z[\Gamma]\rightarrow\mathbb Z[G]\hookrightarrow D$ induced by inclusion $M_{0} \into M$.
    Proceeding as in Section~\ref{primeproof}, we conclude that $H_{<n/2}(M_0;D)=0$.
    So, by Proposition~\ref{finite}, there is a positive integer $m$ (determined by $H_*(X;D[\tau,\tau^{-1}])$, hence by $X,\phi,$ and $\psi$, hence by $M,V_1$ and $V_2$)
    such that for $d$ relatively prime to $m$ the $d$-fold cyclic cover $M'_0\rightarrow M_0$ satisfies $H_{<n/2}(M_0';D)=0$.
    Therefore, as in Section~\ref{primeproof}, the $d$-fold branched cover $\widehat M$ has $H_{\neq n/2}(\widehat M;D)=0$ and consequently $b^{(2)}_{\neq n/2}(\widehat M)=0$.
\end{proof}


\bigskip
\noindent Grigori Avramidi, Max Planck Institute for Mathematics, Bonn, Germany, 53111\\
\url{gavramidi@mpim-bonn.mpg.de} \\

\noindent Boris Okun, University of Wisconsin-Milwaukee, Department of Mathematical Sciences, PO Box 413, Milwaukee, WI 53201-0413\\
\url{okun@uwm.edu} \\

\noindent Kevin Schreve, Louisiana State University, Department of Mathematics, Baton Rouge, LA, 70806 \\
\url{kschreve@lsu.edu}

\end{document}